\documentclass{amsart}
\usepackage{amsmath}
  \usepackage{paralist}
  \usepackage{graphics} 
  \usepackage{epsfig} 
\usepackage{graphicx}  \usepackage{epstopdf}
 \usepackage[colorlinks=true]{hyperref}
\hypersetup{urlcolor=blue, citecolor=red}
\usepackage{dsfont}
\usepackage[utf8]{inputenc}
\usepackage{amssymb}
\usepackage{mathrsfs}
\usepackage{paralist}

\newtheorem{theorem}{Theorem}[section]

\newtheorem{lemma}[theorem]{Lemma}
\newtheorem{proposition}{Proposition}

\theoremstyle{definition}
\newtheorem{definition}[theorem]{Definition}

\newcommand{\Rd}{{\mathbb{R}^d}}
\newcommand{\calH}{{\mathcal{H}}}

\newcommand{\eps}{\varepsilon}
\newcommand{\Rdrei}{{\mathbb{R}^3}}
\newcommand{\W}{\mathbf{W}}

\newcommand{\mom}{\mathbf{m}_2}
\newcommand{\N}{{\mathbb{N}}}
\newcommand{\R}{{\mathbb{R}}}
\newcommand{\flow}[2]{\mathsf{S}_{#1}^{#2}}
\newcommand{\eins}{\mathds{1}}

\newcommand{\dd}{\,\mathrm{d}}
\newcommand{\dn}{\mathrm{d}}
\newcommand{\dv}{\mathrm{div}}
\newcommand{\dff}{\mathrm{D}}
\newcommand{\prb}{\mathscr{P}_2}

\newcommand{\auxil}{\mathcal{A}}
\newcommand{\metr}{\mathbf{d}}
\newcommand{\X}{\mathbf{X}}
\newcommand{\lyp}{\mathcal{L}}
\newcommand{\aU}{\mathcal{U}}
\newcommand{\aV}{\mathcal{V}}
\newcommand{\ent}{\mathcal{E}}
\newcommand{\krnl}{\mathbf{G}}
\newcommand{\dsp}{\mathcal{D}}
\newcommand{\ball}{\mathbb{B}}

\title[Poisson-Nernst-Planck-type system]{Exponential convergence to equilibrium in a Poisson-Nernst-Planck-type system with nonlinear diffusion}

\author[Jonathan Zinsl]{}

\subjclass[2010]{Primary: 35K45; Secondary: 35A15, 35B40, 35D30, 35Q60}
 \keywords{Poisson-Nernst-Planck model, Wasserstein metric, gradient flow, exponential convergence, flow interchange lemma}

 \email{zinsl@ma.tum.de}

\thanks{This research has been supported by the German Research Foundation (DFG), SFB TR 109.}

\begin{document}

\begin{abstract}
We investigate a Poisson-Nernst-Planck type system in three spatial dimensions where the strength of the electric drift depends on a possibly small parameter and the particles are assumed to diffuse quadratically. On grounds of the global existence result proved by Kinderlehrer, Monsaingeon and Xu (2015) using the formal Wasserstein gradient flow structure of the system, we analyse the long-time behaviour of weak solutions. We prove under the assumption of uniform convexity of the external drift potentials that the system possesses a unique steady state. If the strength of the electric drift is sufficiently small, we show convergence of solutions to the respective steady state at an exponential rate using entropy-dissipation methods.
\end{abstract}

\maketitle

\centerline{\scshape Jonathan Zinsl}
\medskip
{\footnotesize
 \centerline{Technische Universit\"at M\"unchen}
\centerline{Zentrum f\"ur Mathematik}
   \centerline{Boltzmannstr. 3}
   \centerline{85747 Garching, Germany}
} 

\medskip

\section{Introduction}\label{sec:intro}
We are concerned with the long-time behaviour of the following parabolic system in three spatial dimensions:
\begin{align}
\partial_t u &= \dv(u\dff(2u+U+\eps \psi)),\label{eq:pnp_u}\\
\partial_t v &= \dv(v\dff(2v+V-\eps \psi)),\label{eq:pnp_v}\\
-\Delta\psi &= u-v,\label{eq:poi}
\end{align}
for \emph{nonnegative} $u,v:[0,\infty)\times\Rdrei\to [0,\infty]$, together with an initial condition $u(0,\cdot)=u^0\ge 0$, $v(0,\cdot)=v^0\ge 0$ on $\Rdrei$. Equations \eqref{eq:pnp_u}\&\eqref{eq:pnp_v} are coupled by means of Poisson's equation \eqref{eq:poi} giving
\begin{align*}
\psi&=\krnl\ast(u-v),
\end{align*}
with \emph{Newton's potential} $\krnl$ in $\Rdrei$, i.e.
\begin{align*}
\krnl(x)=\frac1{4\pi |x|}\quad\text{for }x\neq 0.
\end{align*}
For the external confinement potentials $U,V\in C^2(\Rdrei)$, we assume that they grow quadratically:
\begin{align*}
&\text{Boundedness of second derivatives:}\quad\|\dff^2 U\|_{L^\infty} <\infty,~ \|\dff^2 U\|_{L^\infty}<\infty,\\
&\text{Uniform convexity:}\quad\dff^2 U \ge \lambda_0\eins,~\dff^2 V \ge \lambda_0\eins,
\end{align*}
in the sense of symmetric matrices, for some $\lambda_0>0$.
We may assume without loss of generality that $U,V\ge 0$. Finally, $\eps>0$ is a fixed parameter.

System \eqref{eq:pnp_u}--\eqref{eq:poi} possesses a \emph{gradient flow} structure on the space $\X=\prb\times\prb$, where $\prb$ denotes the space of absolutely continuous probability measures on $\Rdrei$ with finite second moment, endowed with the metric
\begin{align*}
\metr((u,\tilde u),(v,\tilde v)):=\sqrt{\W_2(u,\tilde u)^2+\W_2(v,\tilde v)^2},
\end{align*}
where $\W_2$ is the $L^2$-Wasserstein metric. The corresponding energy functional $\ent:\X\to\R\cup\{+\infty\}$ reads
\begin{align*}
\ent(u,v):=\begin{cases}\int_\Rdrei (u^2+v^2+uU+vV+\frac{\eps}{2}|\dff \psi|^2)\dd x,& \text{if }(u,v)\in L^2\times L^2,\\ +\infty,&\text{otherwise.}\end{cases}
\end{align*}

It has been shown by Kinderlehrer \emph{et al.} \cite{kmx2015} that, given $(u^0,v^0)\in\X\cap (L^2\times L^2)$, a global-in-time weak solution $(u,v):[0,\infty)\to\X$ to \eqref{eq:pnp_u}--\eqref{eq:poi} exists and can be constructed as the continuous-time limit $\tau\to 0$ (in a sense to be specified below) of the minimizing movement scheme with step size $\tau>0$:
\begin{align}
  \label{eq:jko}
\begin{split}
(u_\tau^0,v_\tau^0)&:=(u^0,v^0),\\
  (u_\tau^n,v_\tau^n) &\in \operatorname*{argmin}_{(u,v)\in \X} \Big(\frac1{2\tau}\metr\big((u,v),(u_\tau^{n-1},v_\tau^{n-1})\big)^2 + \ent(u,v)\Big)\quad\text{for }n\in\N.
\end{split}
\end{align}
We will summarize important properties of the sequences $(u_\tau^n,v_\tau^n)_{n\in\N}$ and their limit $(u,v)$ in Section \ref{sec:pre} below.

In this work, we are interested in the behaviour of the aforementioned weak solution $(u,v)$ to system \eqref{eq:pnp_u}--\eqref{eq:poi} as $t\to\infty$. First, we characterize the set of equilibria:
\begin{theorem}[Existence and uniqueness of stationary states]
\label{thm:stat}
For every $\eps>0$, there exists a unique minimizer $(u_\infty,v_\infty)\in (W^{1,2}\times W^{1,2})$ of $\ent$ on $\X$. $(u_\infty,v_\infty)$ is a stationary solution to \eqref{eq:pnp_u}--\eqref{eq:poi} and satisfies
\begin{align}
u_\infty&=\frac12[C_u-U-\eps\psi_\infty]_+,\label{eq:uinf}\\
v_\infty&=\frac12[C_v-V+\eps\psi_\infty]_+,\label{eq:vinf}\\
\psi_\infty&:=\krnl\ast(u_\infty-v_\infty),\nonumber
\end{align}
where $C_u,C_v\in\R$ are such that $\|u_\infty\|_{L^1}=1=\|v_\infty\|_{L^1}$; $[\cdot]_+$ denoting the positive part. For every $\alpha\in(0,1)$, $u_\infty,v_\infty\in C^{0,\alpha}$ with compact support and $\psi\in L^\infty\cap C^{2,\alpha}$.
\end{theorem}

Second, we prove for sufficiently small coupling strength $\eps>0$ exponential convergence to $(u_\infty,v_\infty)$:
\begin{theorem}[Exponential convergence to equilibrium]
\label{thm:conv}
  There are constants $\bar\eps>0$ and $\bar L>0$ such that for all $\delta>0$, there exists $C_\delta>0$ such that the following is true
  for every $\eps\in(0,\bar\eps)$ and arbitrary initial conditions $(u^0,v^0)\in \X\cap(L^2\times L^2)$: 
  The weak solution $(u,v)$ to \eqref{eq:pnp_u}--\eqref{eq:poi} obtained as a limit of the scheme \eqref{eq:jko} converges to $(u_\infty,v_\infty)$
  exponentially fast with rate $\Lambda_\eps:=\lambda_0-\bar L\eps>0$ in the following sense:
  \begin{align}
    \label{eq:conv}
\begin{split}
    \W_2(u(t,\cdot),u_\infty) &+ \W_2(v(t,\cdot),v_\infty)+\|u(t,\cdot)-u_\infty\|_{L^2} + \|v(t,\cdot)-v_\infty\|_{L^2} \\
    &\le C_\delta\left(\ent(u_0,v_0)-\ent(u_\infty,v_\infty)+1\right)^{1+\delta}e^{-\Lambda_\eps t} \qquad \text{for all $t\ge0$}.
\end{split}
  \end{align}
\end{theorem}

System \eqref{eq:pnp_u}-\eqref{eq:poi} may arise as a model for the dynamics of a system consisting of positively and negatively charged particles (e.g. ions) inside some electrically neutral surrounding medium (e.g. air, water). For further details on the mathematical modelling of those phenomena, we refer to the monographs \cite{markowichbook1990,juengelbook2001}. Here, both species are confined by means of external potentials $U$ and $V$ and are assumed to diffuse nonlinearily -- with a diffusive mobility depending linearily on the concentrations $u$ and $v$, respectively. We assume the Poisson coupling by means of equation \eqref{eq:poi} to be suitably weak ($\eps\ll 1$), i.e. the drift induced by electromagnetic force to be small. The quantity $\eps^{-1}\gg 1$ corresponds to a large relative permittivity (dielectric constant) of the surrounding medium. A similar system has been considered by Biler, Dolbeault and Markowich \cite{bilerdolbeault2001}. There, a \emph{time-dependent} coupling $\eps(t)$ was introduced, with the crucial assumption that $\eps(t)\to 0$ as $t\to\infty$, i.e. asymptotical damping of the electrostatic potential. Under relatively general requirements on spatial dimension, external potential and diffusive nonlinearity, convergence to equilibrium as $t\to\infty$ is proved for sufficiently regular solutions. Here, we do \emph{not} require asymptotical damping of the Poisson coupling, that is, the system at hand still constitutes a \emph{coupled} system even in the large-time limit $t\to\infty$. To the best of our knowledge, our rigorous result on exponential convergence of \emph{weak} solutions is novel in the case of genuinely nonlinear diffusion on multiple space dimensions, even in the small coupling regime $\eps\ll 1$. Partial results have been obtained in one spatial dimension \cite{difwunsch2008} or for space-dependent diffusion \cite{benabd2004} only.

In contrast to that, the case of linear diffusion has already been treated almost exhaustively. In the articles \cite{amt2000,bilerdolbeault2000,amtu2001} preceding \cite{bilerdolbeault2001}, it was shown that the rate of exponential convergence to equilibrium of the system without coupling, for uniformly convex potentials, is (almost) retained for coupled systems. There, the strategy of proof is mainly based on applications of \emph{generalized Sobolev inequalities} the derivation of which require the use of a \emph{Holley-Stroock}-type perturbation lemma. Seemingly, such a strategy might not be applicable in the setting of nonlinear diffusion. On the other hand, systems of the form above possess (at least formally) a \emph{gradient flow} structure (w.r.t. e.g. the $L^2$-Wasserstein distance) which also is of use for the analysis of the system -- and, in contrast, does \emph{not} at all require linear diffusion.

Variational techniques related to gradient flows in (transportation) metric spaces \cite{villani2003,savare2008} have been applied to a variety of evolution equations, e.g. in \cite{jko1998, carrillo2000, otto2001, cmv2003, carrillo2006contractions, carrillo2006, agueh2008, gianazza2009, matthes2009}. The variational \emph{minimizing movement scheme} \cite{jko1998} provides a key tool, in combination with generalized convexity assumptions on the respective free energy or driving entropy functionals \cite{mccann1997}, for the investigation of existence and long-time behaviour of solutions to nonlinear evolution equations with gradient structure. Recently, also genuine systems of equations were object of study in this context. Using the minimizing movement scheme on an appropriate metric space, existence of weak solutions has been proved in several cases, e.g. for Keller-Segel-type systems \cite{carrillo2012,blanchet2012,blanchet2014,zinsl2012,zinsl2014,zinsl2015} or others \cite{kinderlehrer2009,laurencot2011,zm2014,kmx2015}. Using the minimizing movement scheme to obtain convergence to equilibrium is rather novel in the case of genuine systems. The method applied here has first been used for Keller-Segel-type models in \cite{zinsl2014,zinsl2015} leading to similar results as Theorem \ref{thm:conv} here.

The basis of our strategy is the fact that the uncoupled system ($\eps=0$) defines a $\lambda_0$-contractive flow, since $\ent$ then is $\lambda_0$-convex along geodesics in $(\X,\metr)$. However, geodesic convexity of $\ent$ is lost if $\eps>0$. Still, one can prove (see \cite{kmx2015}) that a continuous flow in $\X$ of system \eqref{eq:pnp_u}--\eqref{eq:poi} exists and admits a unique steady state $(u_\infty,v_\infty)$ (see our Theorem \ref{thm:stat}). The cornerstone of the proof of Theorem \ref{thm:conv} is the existence of an auxiliary functional $\lyp$ which is ``close'' to $\ent$ (for both $\eps>0$ and $\eps=0$): First, $\lyp$ is $\lambda_\eps$-convex along geodesics in $\X$, with a slightly smaller convexity modulus $0<\lambda_\eps<\lambda_0$, and is decoupled in its arguments $u$ and $v$. Hence, known results on gradient flows for scalar porous-medium type equations apply for the auxiliary gradient flow $\flow{}{\lyp}$ of $\lyp$. This auxiliary flow can be used -- with the almost classical flow interchange technique from \cite{matthes2009} -- to estimate the dissipation of $\lyp$ along the continuous flow given by the free energy $\ent$. We seek to eventually apply Gronwall's lemma for $\lyp$.
Since $\ent$ and $\lyp$ differ by a ``small'' -- but non-convex -- functional, cross-terms occur in the entropy-dissipation estimate and have to be controlled by suitable \emph{a priori} estimates. For small coupling strength, we arrive in the end at an exponential estimate with an again smaller rate $0<\Lambda_\eps<\lambda_\eps<\lambda_0$.

Clearly, this strategy requires $\lambda_0>0$, i.e. uniform convexity of the external potentials $U$ and $V$, which is not needed for proving existence (see \cite{kmx2015}). As in \cite{zinsl2014}, we deal with quadratic diffusion only since the right dissipation estimates do not seem to be at hand in the general case. One last comment is due about the scaling of our exponential estimate \eqref{eq:conv} in Theorem \ref{thm:conv}: For \emph{uniformly contractive} gradient flows, one expects the difference of initial and final energy to enter the estimate with a power $1/2$, corresponding to $\delta=-1/2$. However, due to non-convexity of the free energy $\ent$, only positive $\delta$ can be obtained with our strategy here. Nevertheless, the initial condition only appears via its energy.\\

This paper is organized as follows: First, we recall general facts and definitions for gradient flows in metric spaces and a result on the global existence of solutions to the system at hand. In Section \ref{sec:stat}, we prove Theorem \ref{thm:stat} on existence and uniqueness of steady states. Section \ref{sec:auxent} is devoted to the introduction and investigation of the auxiliary entropy mentioned above. There, we also derive a central entropy-dissipation estimate for our forthcoming analysis, using the flow interchange technique. Finally, Theorem \ref{thm:conv} is proved in Section \ref{sec:conv}. There, we first derive an additional \emph{a priori} estimate on the auxiliary entropy holding for large times. In consequence, exponential convergence is proved.


%
\section{Preliminaries}\label{sec:pre}
\subsection{Geodesic convexity and gradient flows}
In this section, we will briefly mention relevant definitions and facts on gradient flows in metric spaces. For a more thorough presentation, we refer to \cite{savare2008, villani2003}.

Throughout this paper, $\dff$ and $\dff^2$ denote the spatial gradient and Hessian, respectively. By abuse of notation, we often identify an absolutely continuous measure with its Lebesgue density. A sequence $(\mu_n)_{n\in\N}$ of probability measures on $\R^d$ is said to \emph{converge narrowly} to some limit probability measure $\mu$ if for all continuous and bounded maps $\phi:\R^d\to\R$, one has
\begin{align*}
\lim_{n\to\infty}\int_\Rd \phi(x)\dd\mu_n(x)&=\int_\Rd \phi(x)\dd \mu(x).
\end{align*}
For the metric space $(\prb,\W_2)$, the following is true: A sequence $(\mu_n)_{n\in\N}$ in $\prb$ converges w.r.t. the $L^2$-Wasserstein distance $\W_2$ if and only if \emph{both} $\mu_n\rightharpoonup \mu$ narrowly \emph{and} the sequence of second moments converges:
\begin{align*}
\lim_{n\to\infty}\mom(\mu_n)=\mom(\mu),\quad\text{with }\mom(\rho):=\int_\Rd |x|^2\dd\rho(x)\text{ for }\rho\in\prb.
\end{align*}

A functional $\auxil:X\to\R\cup\{\infty\}$ defined on some metric space $(X,d)$ is called $\lambda$-\emph{geodesically convex}
for some $\lambda\in\R$ if for every $w_0,w_1\in X$ and $s\in [0,1]$, one has
\begin{align*}
 \auxil(w_s)\le (1-s)\auxil(w_0)+s\auxil(w_1)-\frac{\lambda}{2}s(1-s)d^2(w_0,w_1),
\end{align*}
where $w_s:\,[0,1]\to X,\,s\mapsto w_s$ is a \emph{geodesic} curve connecting $w_0$ and $w_1$.
We recall two important classes of $\lambda$-convex functionals (see e.g. \cite[Ch. 9.3]{savare2008}, {\cite[Thm. 5.15]{villani2003}}) on the space $(\prb,\W_2)$:

\begin{theorem}[Criteria for geodesic convexity on $(\prb,\W_2)$]\label{thm:crit_conv}
The following statements are true:
 \begin{enumerate}[(a)]
 \item Let $h\in C^0([0,\infty))$ be given, and define a functional $\mathcal{A}$ on $\prb$ by $\mathcal{A}(w):=\int_\Rd h(w(x))\dd x$ for $w\in\prb\cap L^1$.
   If $h(0)=0$ and $r\mapsto r^dh(r^{-d})$ is convex and nonincreasing on $(0,\infty)$,
 $\mathcal{A}$ is $0$-geodesically convex and lower semicontinuous in $(\prb,\W_2)$.
 \item Let a function $W\in C^0(\Rd)$ be given, and define a functional $\mathcal{A}(\mu):=\int_\Rd W\dd\mu$ for all $\mu\in\prb$.
   If $W$ is $\lambda$-convex for some $\lambda\in\R$,
 $\mathcal{A}$ is $\lambda$-geodesically convex in $(\prb,\W_2)$.
 \end{enumerate}
\end{theorem}

As to the notion of \emph{gradient flow}, we use the following characterization:
\begin{definition}[$\kappa$-contractive flow]
 Let $\auxil:X\to\R\cup\{\infty\}$ be a lower semicontinuous functional on the metric space $(X,d)$.
 A continuous semigroup $\flow{}{\auxil}$ on $(X,d)$ is called \emph{$\kappa$-flow} for some $\kappa\in\R$,
 if the \emph{evolution variational estimate}
 \begin{align*}
   \frac{1}{2}\frac{\dn^+}{\dn t}d^2(\flow{t}{\auxil}(w),\tilde w)+\frac{\kappa}{2}d^2(\flow{t}{\auxil}(w),\tilde w)+\auxil(\flow{t}{\auxil}(w))
   &\le\auxil(\tilde w)
\end{align*}
 holds for arbitrary $w,\tilde w$ in the domain of $\auxil$, and for all $t\ge 0$.
\end{definition}
We recall some facts on gradient flows of convex functionals on $\prb$:
\begin{theorem}[Gradient flows of geodesically convex functionals on $(\prb,\W_2)$ \cite{savare2008}]\label{thm:gfcoll}
Let $\auxil:\prb\to\R\cup\{\infty\}$ be lower semicontinuous and $\lambda$-geodesically convex w.r.t. the distance $\W_2$. The following statements hold:
\begin{enumerate}[(a)]
\item There exists a unique $\kappa$-flow, with $\kappa:=\lambda$, for $\auxil$. Its corresponding evolution equation can be written as
\begin{align*}
   \partial_t \flow{t}{\auxil}(w) &=\mathrm{div}\left(\flow{t}{\auxil}(w) \dff \left(\frac{\delta\auxil}{\delta w}(\flow{t}{\auxil}(w))\right)\right),
 \end{align*}
if $\auxil$ is sufficiently regular. There, $\frac{\delta\auxil}{\delta w}$ stands for the usual first variation of the functional $\auxil$ on $L^2$.
\item There exists exactly one minimizer $w_{\min}$ of $\auxil$, for which the following holds:
\begin{align}
\frac{\lambda}{2}\W_2^2(w,w_{\min})&\le \auxil(w)-\auxil(w_{\min})\le \frac{1}{2\lambda}\lim_{h\searrow 0}\frac{\auxil(w)-\auxil(\flow{h}{\auxil}(w))}{h}.\label{eq:subdiff}
\end{align}
\end{enumerate}
\end{theorem}
One of the cornerstones of our analysis below is the following theorem:
\begin{theorem}[Flow interchange lemma {\cite[Thm. 3.2]{matthes2009}}]\label{thm:flowinterchange}
Let $\mathcal{B}$ be a proper, lower semicontinuous and $\lambda$-geodesically convex functional on $(X,d)$. Let furthermore $\mathcal{A}$ be another proper, lower semicontinuous functional on $(X,d)$ such that $\operatorname{Dom}(\mathcal{A})\subset \operatorname{Dom}(\mathcal{B})$. Assume that, for arbitrary $\tau>0$ and $\tilde w\in X$, the functional $\frac1{2\tau}d(\cdot,\tilde w)^2+\mathcal{A}$ possesses a minimizer $w$ on $X$. Then, the following holds:
\begin{align*}
\mathcal{B}(w)+\tau \mathrm{D}^\mathcal{B}\mathcal{A}(w)+\frac{\lambda}{2}d^2(w,\tilde w)&\le \mathcal{B}(\tilde w).
\end{align*}
There, $\mathrm{D}^\mathcal{B}\mathcal{A}(w)$ denotes the \emph{dissipation} of the functional $\mathcal{A}$ along the gradient flow $\flow{}{\mathcal{B}}$ of the functional $\mathcal{B}$, i.e.
\begin{align*}
\mathrm{D}^\mathcal{B}\mathcal{A}(w):=\limsup_{h\searrow 0}\frac{\mathcal{A}(w)-\mathcal{A}(\flow{h}{\mathcal{B}}(w)    )   }{h}.
\end{align*}
\end{theorem}

\subsection{Minimizing movement and existence of solutions}
In this subsection, we recall the results proved by Kinderlehrer \emph{et al.} in \cite{kmx2015} in our specific setting. 
\begin{proposition}[Minimizing movement {\cite[Prop.~3.3]{kmx2015}}]\label{prop:minmov}
Let $\tau>0$ and $(u^0,v^0)\in\X\cap (L^2\times L^2)$ be given. Then, the sequence $(u_\tau^n,v_\tau^n)_{n\in\N}$ defined by the minimizing movement scheme \eqref{eq:jko} is well-defined with $(u_\tau^n,v_\tau^n)\in\X\cap (W^{1,2}\times W^{1,2})$ for all $n\in\N$. By definition, the sequence $(\ent(u_\tau^n,v_\tau^n))_{n\in\N}$ is nonincreasing.
\end{proposition}

Define for $\tau>0$ the \emph{discrete solution} $(u_\tau,v_\tau):[0,\infty)\to \X$ by piecewise constant interpolation, that is
\begin{align}\label{eq:disc_sol}
\begin{split}
(u_\tau,v_\tau)(0)&:=(u^0,v^0),\\
(u_\tau,v_\tau)(t)&:=(u_\tau^n,v_\tau^n)\text{ for }t\in ((n-1)\tau,n\tau]\text{ and }n\ge 1.
\end{split}
\end{align}

The following main result of \cite{kmx2015} about the existence of solutions to \eqref{eq:pnp_u}--\eqref{eq:poi} is at the basis of our subsequent analysis:
\begin{theorem}[Existence of solutions {\cite[Thm.~2]{kmx2015}}]\label{thm:exist}
Let $\eps>0$ and $U,V$ as mentioned above be given. Define, for initial conditions $(u^0,v^0)\in\X\cap(L^2\times L^2)$ and each $\tau>0$ a discrete solution $(u_\tau,v_\tau)$ by \eqref{eq:jko}\&\eqref{eq:disc_sol}. Then, there exists a sequence $\tau_k\searrow 0$ and a map $(u,v):[0,\infty)\times\R^3\to[0,\infty]^2$ such that for each $t>0$, $u_{\tau_k}(t)\rightharpoonup u(t)$ and $v_{\tau_k}(t)\rightharpoonup v(t)$, both narrowly in $\prb$ as $k\to\infty$. Moreover, $(u,v)$ is a solution to \eqref{eq:pnp_u}--\eqref{eq:poi} in the sense of distributions, it attains the initial condition and one has for each $T>0$:
\begin{align*}
u,v&\in C^{1/2}([0,T];(\prb,\W_2))\cap L^\infty([0,T];L^2)\cap L^2([0,T];W^{1,2}),\\
\ent(u(T),v(T))&\le \ent(u^0,v^0).
\end{align*}
\end{theorem}

%
\section{The equilibrium state}\label{sec:stat}
In this section, we prove Theorem \ref{thm:stat}.
\begin{proof}
\underline{Existence:} 
Trivially, $\ent$ is bounded from below. Hence, there exists a minimizing sequence $(u_k,v_k)_{k\in\N}$ in $\X\cap(L^2\times L^2)$ with $\lim\limits_{k\to\infty}\ent(u_k,v_k)=\inf\limits_{(u,v)\in\X}\ent(u,v)$. Thus, we have for some $C>0$ that $\|u_k\|_{L^2}\le C$, $\|v_k\|_{L^2}\le C$ for all $k\in\N$. Moreover, using the $\lambda_0$-convexity of $U$ and $V$ on $\R^3$, one obtains $\sup\limits_{k\in\N}\mom(u_k)<\infty$ and $\sup\limits_{k\in\N}\mom(v_k)<\infty$ with the help of the elementary estimates $U(x)-U(x_{\min}^U)\ge \frac{\lambda_0}{4}|x|^2-\frac{\lambda_0}{2}|x_{\min}^U|^2$ and $V(x)-V(x_{\min}^V)\ge \frac{\lambda_0}{4}|x|^2-\frac{\lambda_0}{2}|x_{\min}^V|^2$ (with the unique minimizers $x_{\min}^U, x_{\min}^V$ of $U$ and $V$ on $\R^3$, respectively). We infer with the Prokhorov and Banach-Alaoglu theorems that there exists a subsequence (non-relabelled) and a limit $(u_\infty,v_\infty)\in\X\cap(L^2\times L^2)$ such that $u_k\rightharpoonup u_\infty$ and $v_k\rightharpoonup v_\infty$ both narrowly as probability measures and weakly in $L^2$, as $k\to\infty$. With respect to these convergences, $\ent$ is lower semicontinuous. In fact, this is obvious for the quadratic and linear terms in $\ent$ since $U$ and $V$ grow quadratically. For the last term containing the Dirichlet energy $\frac12\|\dff\psi\|_{L^2}^2$, we refer to \cite[Prop.~6.1]{kmx2015} for a result on lower semicontinuity w.r.t. weak $L^1$ convergence. Hence, it follows that $(u_\infty,v_\infty)$ is indeed a minimizer of $\ent$ on $\X$ and hence also a steady state of \eqref{eq:pnp_u}--\eqref{eq:poi}.\\
\underline{Uniqueness:} 
We claim that $\ent$ is uniformly convex with respect to the flat distance induced by the product norm $\|\cdot\|_{L^2\times L^2}$, which implies the uniqueness of minimizers. Indeed, for all $(u,v),(u',v')\in \X\cap(L^2\times L^2)$ and all $s\in [0,1]$, we have, thanks to 
\begin{align}
\label{eq:psipi}
\int_\Rdrei |\dff(\krnl\ast w)|^2\dd x&=\int_\Rdrei (\krnl\ast w)w\dd x=\int_\Rdrei\int_\Rdrei w(x)\krnl(x-y)w(y)\dd x\dd y,
\end{align}
which holds for all $w\in \X\cap L^2$, that
\begin{align*}
&\frac{\dd^2}{\dd s^2}\bigg.\bigg|_{s=0}\ent(u+s(u'-u),v+s(v'-v))\\&=\int_\Rdrei \left[2(u'-u)^2+2(v'-v)^2+\eps((u'-u)-(v'-v))\krnl\ast((u'-u)-(v'-v))\right]\dd x\\&\ge 2\|(u'-u,v'-v)\|_{L^2\times L^2}^2,
\end{align*}
so $\ent$ is $2$-convex w.r.t. the distance induced by $\|\cdot\|_{L^2\times L^2}$.\\
\underline{Euler-Lagrange equations:} 
Since $(u_\infty,v_\infty)$ is the minimizer of $\ent$, the following variational inequality holds:
\begin{align}\label{eq:varineq}
\begin{split}
0&\le \frac{\dd^+}{\dd s}\bigg.\bigg|_{s=0}\ent(u_\infty+s\tilde u,v_\infty+s\tilde v)\\
&=\int_\Rdrei\left[2u_\infty+U+\eps \krnl\ast(u_\infty-v_\infty)\right]\tilde u\dd x\\&\quad+\int_\Rdrei\left[2v_\infty-V-\eps \krnl\ast(u_\infty-v_\infty)\right]\tilde v\dd x,
\end{split}
\end{align}
for all $\tilde u,\tilde v$ such that both $u_\infty+\tilde u\ge 0$ and $v_\infty+\tilde v\ge 0$ on $\Rdrei$, and $\int_\Rdrei \tilde u\dd x=0=\int_\Rdrei \tilde v\dd x$. In order to prove \eqref{eq:uinf}, we set $\tilde v:=0$. Let $\phi:\R^3\to\R$ be such that $\int_\Rdrei \phi\dd x\le 1$ and $\phi+u_\infty\ge 0$ on $\Rdrei$. The choice
\begin{align*}
\tilde u_\phi:=\frac12 \phi-\frac12 u_\infty\int_\Rdrei\phi\dd x
\end{align*}
is admissible for $\tilde u$ in \eqref{eq:varineq}, hence (recall our notation $\psi_\infty:=\krnl\ast(u_\infty-v_\infty)$)
\begin{align}
\label{eq:varu}
0&\le \int_\Rdrei (2u_\infty+U+\eps\psi_\infty-C_u)\phi\dd x,
\end{align}
with
\begin{align*}
C_u&:=\int_\Rdrei (2u_\infty^2+Uu_\infty+\eps u_\infty\psi_\infty)\dd x\in\R.
\end{align*}
If $u_\infty(x)>0$ for some $x\in\R^3$, we are able to choose $\phi$ supported on a small neighborhood of $x$ and to replace by $-\phi$ in \eqref{eq:varu} and obtain
\begin{align*}
2u_\infty(x)+U(x)+\eps\psi_\infty(x)&=C_u.
\end{align*}
If $u_\infty(x)=0$ for some $x$, one has $U(x)-\eps\psi_\infty(x)-C_u\ge 0$, and hence \eqref{eq:uinf} is true in both cases. The equation for $v_\infty$ \eqref{eq:vinf} can be derived in analogy.\\
\underline{Properties:} 
First, since $(u_\infty,v_\infty)$ are admissible as starting condition $(u_\tau^0,v_\tau^0)$ (for arbitrary $\tau>0$) in scheme \eqref{eq:jko}, we obtain thanks to the minimizing property and Proposition \ref{prop:minmov} that $(u_\infty,v_\infty)\in W^{1,2}\times W^{1,2}$. We now show that $\psi_\infty\in L^\infty$. To this end, let $x\in\R^3$ and observe at first that
\begin{align*}
\int_{\ball_1(x)}\frac{|u_\infty(y)-v_\infty(y)|}{|x-y|}\dd y&\le \|u_\infty-v_\infty\|_{L^2}\left(\int_{\ball_1(x)}\frac1{|x-y|^2}\dd y\right)^{1/2}\\&=2\sqrt{\pi}\|u_\infty-v_\infty\|_{L^2},
\end{align*}
independent of $x$, by H\"older's inequality and the transformation theorem. Furthermore, since $|x-y|\ge 1$ if $y\notin\ball_1(x)$ and $\|u_\infty\|_{L^1}=1=\|v_\infty\|_{L^1}$, we get
\begin{align*}
\int_{\Rdrei\setminus\ball_1(x)}\frac{|u_\infty(y)-v_\infty(y)|}{|x-y|}\dd y&\le \|u_\infty-v_\infty\|_{L^1} \sup_{y\notin\ball_1(x)}|x-y|^{-1}\le 2.
\end{align*}
Putting both parts together, we see that $\sup\limits_{x\in\Rdrei}|\psi_\infty(x)|<\infty$. In view of \eqref{eq:uinf}\&\eqref{eq:vinf}, $\psi_\infty\in L^\infty$ implies that $u_\infty$ and $v_\infty$ have compact support since $U$ and $V$ grow quadratically as $|x|\to\infty$. By classical results on solutions to Poisson's equation \cite[Thm.~10.2]{lieb2001}, we then infer that $\psi_\infty\in C^{0,\alpha}$ for all $\alpha\in (0,1)$, since by the Gagliardo-Nirenberg-Sobolev inequality, one has $(u_\infty,v_\infty)\in L^6\times L^6$. Hence, using \eqref{eq:uinf}\&\eqref{eq:vinf} again, we conclude that $u_\infty$ and $v_\infty$ also are H\"older continuous. By elliptic regularity for Poisson's kernel \cite[Thm.~10.3]{lieb2001}, it follows that $\psi_\infty\in C^{2,\alpha}$.
\end{proof}

%
\section{Auxiliary entropy and dissipation}\label{sec:auxent}
In this section, we define a suitable geodesically convex auxiliary entropy $\lyp$ and derive the dissipation of the driving entropy $\ent$ along the gradient flow $\flow{}{\lyp}$ of $\lyp$. \\

Let $\lyp:\X\to\R\cup\{\infty\}$ be defined via
\begin{align*}
\lyp(u,v):=
\begin{cases}
\int_{\Rdrei}\left[\right.u^2-u_\infty^2+v^2-v_\infty^2+(u-u_\infty)U+(v-v_\infty)V\\\quad+\eps(u-u_\infty)\psi_\infty-\eps(v-v_\infty)\psi_\infty\left.\right]\dd x,\\\qquad\qquad\text{if }(u,v)\in L^2\times L^2,\\
+\infty,\qquad\qquad\text{otherwise}.
\end{cases}
\end{align*}

Obviously, $\lyp$ is proper and lower semicontinuous on $(\X,\metr)$. 

\begin{proposition}[Properties of $\lyp$]\label{prop:lyp}
There exists $\eps_0>0$ such that for all $\eps\in (0,\eps_0)$, the following statements hold:
\begin{enumerate}[(a)]
\item There exists $L>0$ such that $\lyp$ is $\lambda_\eps$-geodesically convex w.r.t. $\metr$, where $\lambda_\eps:=\lambda_0-L\eps>0$.
\item The following holds for all $(u,v)\in \X\cap(W^{1,2}\times W^{1,2})$:
\begin{align}
\label{eq:lypest}
\begin{split}
&\quad\|u-u_\infty\|_{L^2}^2+\|v-v_\infty\|_{L^2}^2\\&\le \lyp(u,v)\\&\le \frac1{2\lambda_\eps}\int_\Rdrei\left[u|\dff(2u+U+\eps\psi_\infty)|^2+v|\dff(2v+V-\eps\psi_\infty)|^2\right]\dd x.
\end{split}
\end{align}
\item There exists a constant $K>0$ independent of $\eps$ such that for all $(u,v)\in\X$:
\begin{align}
\label{eq:LE}
\lyp(u,v)&\le \ent(u,v)-\ent(u_\infty,v_\infty)+K\eps.
\end{align}
\end{enumerate}
\end{proposition}

\begin{proof}
\begin{enumerate}[(a)]
\item In view of Theorem \ref{thm:crit_conv}, as $\lyp$ is decoupled in its arguments $u$ and $v$, it suffices to prove that there exists $C>0$ such that $\|\dff^2\psi_\infty\|_{L^\infty}\le C$ for all sufficiently small $\eps>0$. Let $R>0$ such that $\operatorname{supp} u_\infty\cup \operatorname{supp} v_\infty\subset \ball_R(0)$. Since $\psi_\infty\in C^2$ thanks to Theorem \ref{thm:stat}, we have $\sup\limits_{x\in \overline{\ball_{R+1}(0)}}|\partial_{x_i}\partial_{x_j}\psi_\infty(x)|<\infty$ for each pair $(i,j)\in \{1,2,3\}^2$. Consider now $x\notin  \overline{\ball_{R+1}(0)}$. One easily obtains for $z\neq 0$ that 
\begin{align*}
\partial_{z_i}\partial_{z_j}\krnl(z)&=\frac1{4\pi|z|^3}\left(\frac{3z_iz_j}{|z|^2}-\delta_{ij}\right),
\end{align*}
where $\delta_{ij}$ denotes Kronecker's delta. So, using a linear transformation,
\begin{align*}
|\partial_{x_i}\partial_{x_j}\psi_\infty(x)|&=\left|\int_{\ball_R(x)}\partial_{z_i}\partial_{z_j}\krnl(z)(u_\infty(x-z)-v_\infty(x-z))\dd z\right|\\&\le \frac13 R^3\|u_\infty-v_\infty\|_{L^\infty},
\end{align*}
since for all $z\in \ball_R(x)$, one has $|z|>1$ by definition of $x$. Hence, the desired uniform estimate on $\dff^2 \psi_\infty$ is proved.
\item The upper estimate is a straightforward consequence of $\lambda_\eps$-convexity of $\lyp$ and the structure of its Wasserstein subdifferential w.r.t. $u$ and $v$, respectively (see e.g. \cite[Lemma 10.4.1]{savare2008}), in combination with \eqref{eq:subdiff}. For the lower estimate, we observe that
\begin{align*}
&\quad\lyp(u,v)\\&=\int_\Rdrei \left[\right.(u-u_\infty)^2+(v-v_\infty)^2+(u-u_\infty)(2u_\infty+U+\eps\psi_\infty)\\&\quad+(v-v_\infty)(2v_\infty+V-\eps\psi_\infty)\left.\right]\dd x.
\end{align*}
We prove that $\int_\Rdrei(u-u_\infty)(2u_\infty+U+\eps\psi_\infty)\dd x\ge 0$. Since the last term above can be treated in the same way, the claim then follows. Using \eqref{eq:uinf}, we obtain
\begin{align*}
&\quad\int_\Rdrei(u-u_\infty)(2u_\infty+U+\eps\psi_\infty)\dd x\\&=\int_{\{C_u-U-\eps\psi_\infty>0\}}(u-u_\infty)C_u\dd x+\int_{\{C_u-U-\eps\psi_\infty\le 0\}}u(U+\eps\psi_\infty)\dd x\\
&=C_u\int_\Rdrei (u-u_\infty)\dd x+\int_{\{C_u-U-\eps\psi_\infty\le 0\}} u(U+\eps\psi_\infty-C_u)\dd x\\&\ge 0,
\end{align*}
since $u$ and $u_\infty$ have equal mass (hence the first term is equal to zero) and the integrand of the second integral is nonnegative on the domain of integration.
\item One has for all $(u,v)\in \X\cap (L^2\times L^2)$:
\begin{align*}
&\quad\frac1{\eps}(\lyp(u,v)-\ent(u,v)+\ent(u_\infty,v_\infty))\\&=\int_\Rdrei\left[(u-u_\infty)\psi_\infty-(v-v_\infty)\psi_\infty-\frac12|\dff\psi|^2+\frac12|\dff\psi_\infty|^2\right]\dd x\\
&\le \int_\Rdrei \psi_\infty(u-v-\frac12 u_\infty+\frac12 v_\infty)\dd x\le 3\|\psi_\infty\|_{L^\infty}\\&\le K,
\end{align*}
thanks to \eqref{eq:psipi} and Theorem \ref{thm:stat}.
\end{enumerate}
\end{proof}

According to Theorem \ref{thm:gfcoll}(a), the $\lambda_\eps$-contractive flow $\flow{}{\lyp}=:(\aU,\aV)$ is characterized by
\begin{align}
\begin{split}
\partial_s \aU&=\dv\left[\aU\dff(2\aU+U+\eps\psi_\infty)\right],\\
\partial_s \aV&=\dv\left[\aV\dff(2\aV+V-\eps\psi_\infty)\right].
\end{split}
\label{eq:lypflow}
\end{align}

Now, we derive the central \emph{a priori} estimate on the discrete solution:

\begin{proposition}[Dissipation of $\ent$ along $\flow{}{\lyp}$]\label{prop:dsp}
Let $\tau>0$ and let $(u_\tau^n,v_\tau^n)_{n\in\N}$ be the sequence defined via the minimizing movement scheme \eqref{eq:jko}. Then, for all $n\in\N$:
\begin{align}
\label{eq:fint}
\lyp(u_\tau^n,v_\tau^n)+\tau\dsp(u_\tau^n,v_\tau^n)&\le \lyp(u_\tau^{n-1},v_\tau^{n-1}),
\end{align}
the \emph{dissipation} being given by
\begin{align}
\label{eq:dsp}
\begin{split}
\dsp(u,v)&:=\left(1-\frac{\eps}{2}\right)\int_\Rdrei\big(u|\dff(2u+U+\eps\psi_\infty)|^2+v|\dff(2v+V-\eps\psi_\infty)|^2\big)\dd x\\&\quad-\frac{\eps}{2}\int_\Rdrei (u+v)|\dff(\psi-\psi_\infty)|^2\dd x.
\end{split}
\end{align}
\end{proposition}

\begin{proof}
To justify the calculations below, we regularize the flow given by \eqref{eq:lypflow}. Define, for $\nu>0$ and $(u,v)\in\X\cap (L^2\times L^2)$ the regularized functional
\begin{align*}
\lyp_\nu(u,v)&:=\lyp(u,v)+\nu\calH(u)+\nu\calH(v),
\end{align*}
with \emph{Boltzmann's entropy} $\calH(w):=\int_\Rdrei w\log w\dd x$, which is finite on $\prb\cap L^2$ (cf. e.g. \cite[Lemma 5.3]{zinsl2012}). Furthermore, by Theorem \ref{thm:gfcoll}(a), $\calH$ is $0$-geodesically convex on $\prb$, so $\lyp_\nu$ is $\lambda_\eps$-geodesically convex w.r.t. $\metr$ and the associated evolution equation to its $\lambda_\eps$-flow $(\aU,\aV)$ is the strictly parabolic, decoupled system
\begin{align}
\begin{split}
\partial_s \aU&=\nu\Delta\aU+\dv\left[\aU\dff(2\aU+U+\eps\psi_\infty)\right],\\
\partial_s \aV&=\nu\Delta\aV+\dv\left[\aV\dff(2\aV+V-\eps\psi_\infty)\right].
\end{split}
\label{eq:nuflow}
\end{align}
Let $(u,v)\in \X\cap(W^{1,2}\times W^{1,2})$. At least for small $s>0$, system \eqref{eq:nuflow} has a smooth and nonnegative solution $(\aU,\aV)$ such that $(\aU(s),\aV(s))\rightarrow (u,v)$ both strongly in $L^2\times L^2$ and $\metr$, as well as weakly in $W^{1,2}\times W^{1,2}$, for $s\searrow 0$. Moreover, this local flow can be identified with the $\lambda_\eps$-flow associated to $\lyp_\nu$ (see e.g. \cite[Thm.~11.2.8]{savare2008}). Then, writing $\Psi:=\krnl\ast(\aU-\aV)$ for brevity:
\begin{align*}
-\frac{\dd}{\dd s}\ent(\aU,\aV)=&-\int_\Rdrei [2\aU+U+\eps\Psi]\dv\left[\nu\dff\aU+\aU\dff(2\aU+U+\eps\psi_\infty)\right]\dd x\\&-\int_\Rdrei[2\aV+V-\eps\Psi]\dv\left[\nu\dff\aV+\aV\dff(2\aV+V-\eps\psi_\infty)\right]\dd x.
\end{align*}
We first focus on the viscosity terms and obtain, using that $(\aU,\aV)\in\X$:
\begin{align*}
&\quad-\int_\Rdrei \big([2\aU+U+\eps\Psi]\Delta\aU+[2\aV+V-\eps\Psi]\Delta\aV\big)\dd x\\
&=\int_\Rdrei\big(2|\dff\aU|^2+2|\dff\aV|^2-\aU\Delta U-\aV\Delta V-\eps(\aU-\aV)\Delta\Psi\big)\dd x\\
&= 2\|\dff\aU\|_{L^2}^2+2\|\dff\aV\|_{L^2}^2-\int_\Rdrei(\aU\Delta U+\aV\Delta V)\dd x+\eps\|\aU-\aV\|_{L^2}^2\\
&\ge -\|\Delta U\|_{L^\infty}-\|\Delta V\|_{L^\infty}.
\end{align*}

The remaining terms can be rewritten as
\begin{align*}
&\quad-\int_\Rdrei [2\aU+U+\eps\Psi]\dv\left[\aU\dff(2\aU+U+\eps\psi_\infty)\right]\dd x\\&\qquad-\int_\Rdrei[2\aV+V-\eps\Psi]\dv\left[\aV\dff(2\aV+V-\eps\psi_\infty)\right]\dd x\\
&=\int_\Rdrei \aU|\dff(2\aU+U+\eps\psi_\infty)|^2\dd x+\int_\Rdrei \aV|\dff(2\aV+V-\eps\psi_\infty)|^2\dd x\\
&\quad+\eps\int_\Rdrei \aU \dff(2\aU+U+\eps\psi_\infty)\cdot\dff(\Psi-\psi_\infty)\dd x\\&\quad-\eps\int_\Rdrei \aV \dff(2\aV+V-\eps\psi_\infty)\cdot\dff(\Psi-\psi_\infty)\dd x\\
&\ge \left(1-\frac{\eps}{2}\right)\int_\Rdrei\big(\aU|\dff(2\aU+U+\eps\psi_\infty)|^2+\aV|\dff(2\aV+V-\eps\psi_\infty)|^2\big)\dd x\\&\quad-\frac{\eps}{2}\int_\Rdrei (\aU+\aV)|\dff(\Psi-\psi_\infty)|^2\dd x,
\end{align*}
using Young's inequality in the final step. All in all, we arrive at
\begin{align*}
-\frac{\dd}{\dd s}\ent(\aU,\aV)&\ge \dsp(\aU,\aV)-\nu\big(\|\Delta U\|_{L^\infty}+\|\Delta V\|_{L^\infty}\big).
\end{align*}
Observing that the terms appearing in $\dsp$ are lower semicontinuous w.r.t. the convergence of $(\aU,\aV)\to(u,v)$ above, we obtain after passage to the limits $s\searrow 0$ and $\nu\searrow 0$ that $\dff^{\lyp}\ent(u,v)\ge \dsp(u,v)$.
Application of the flow interchange lemma (Theorem \ref{thm:flowinterchange}) completes the proof of \eqref{eq:fint}.
\end{proof}

The remaining task will be to establish appropriate bounds on the dissipation $\dsp(u_\tau^n,v_\tau^n)$ in terms of $\lyp(u_\tau^n,v_\tau^n)$ in order to apply a discrete Gronwall lemma and to conclude exponential convergence. Note that, in view of \eqref{eq:subdiff}, it will be enough to control the second part of $\dsp(u_\tau^n,v_\tau^n)$.

%
\section{Convergence to equilibrium}\label{sec:conv}
In this section, we complete the proof of Theorem \ref{thm:conv}. Our strategy is as follows: First, we derive a uniform bound (independent of $\eps$ and the initial condition) on the auxiliary entropy $\lyp$ for sufficiently large times. This brings us into position to prove a refined estimate on the dissipation $\dsp$ strong enough to infer exponential convergence of $\lyp$ to zero.

In the following, for $\tau>0$, we denote by $(u_\tau^n,v_\tau^n)_{n\in\N}$ a sequence given by the minimizing movement scheme \eqref{eq:jko}.
\subsection{Boundedness of auxiliary entropy}
We first need an additional estimate for the dissipation terms in $\eqref{eq:dsp}$.
\begin{lemma}
There exists a constant $\theta>0$ such that for all $\eps\in(0,\eps_0)$ and all $(u,v)\in\X\cap (W^{1,2}\times W^{1,2})$:
\begin{align}
\label{eq:L3D}
\begin{split}
\|u\|_{L^3}^4&\le \theta\left(1+\int_{\Rdrei}u|\dff(2u+U+\eps\psi_\infty)|^2\dd x\right),\\
\|v\|_{L^3}^4&\le \theta\left(1+\int_{\Rdrei}v|\dff(2v+V-\eps\psi_\infty)|^2\dd x\right),
\end{split}
\end{align}
with the convention that the respective right-hand side is equal to $+\infty$ if \break $u|\dff(2u+U+\eps\psi_\infty)|^2$ or $v|\dff(2v+V-\eps\psi_\infty)|^2$ is not integrable.
\end{lemma}

\begin{proof}
We shall prove the statement for $u$; the other one can be shown analogously. We assume that the r.h.s. is finite. Expanding the square and integrating by parts, one has
\begin{align*}
&\quad\int_\Rdrei u|\dff(2u+U+\eps\psi_\infty)|^2\dd x\\&=\int_\Rdrei \left(\frac{16}{9}|\dff u^{3/2}|^2-2u^2\Delta(U+\eps\psi_\infty)+u|\dff(U+\eps\psi_\infty)|^2\right)\dd x.
\end{align*}
Since $\Delta U$ and $\Delta\psi_\infty=v_\infty-u_\infty$ are essentially bounded, we obtain
\begin{align*}
\frac{16}{9}\|\dff u^{3/2}\|_{L^2}^2&\le \int_\Rdrei u|\dff(2u+U+\eps\psi_\infty)|^2\dd x+C\|u\|_{L^2}^2,
\end{align*}
for some constant $C>0$. By the triangle and Young inequalities, one has $\|u\|_{L^2}^2\le 2\|u_\infty\|_{L^2}^2+2\|u-u_\infty\|_{L^2}^2$. For small $\eps>0$, we can use \eqref{eq:lypest} and arrive at
\begin{align*}
\frac{16}{9}\|\dff u^{3/2}\|_{L^2}^2&\le \int_\Rdrei \left(1+\frac{C}{\lambda_\eps}\right) u|\dff(2u+U+\eps\psi_\infty)|^2\dd x +2C\|u_\infty\|_{L^2}^2.
\end{align*}
On the other hand, with the $L^p$-interpolation and Gagliardo-Nirenberg-Sobolev inequalities, we have (recall $\|u\|_{L^1}=1$):
\begin{align*}
\|u\|_{L^3}&\le \|u\|_{L^9}^{3/4}\|u\|_{L^1}^{1/4}=\|u^{3/2}\|_{L^6}^{1/2}\le C'\|\dff u^{3/2}\|_{L^2}^{1/2}.
\end{align*}
Raising to the fourth power, we end up with \eqref{eq:L3D}.
\end{proof}

We now derive a uniform bound on $\lyp$ for large times.

\begin{proposition}[Boundedness of $\lyp$]\label{prop:lypbd}
\begin{enumerate}[(a)]
\item There exist $\eps_1\in(0,\eps_0)$, $L'>0$ and $M>0$ such that for all $\eps\in(0,\eps_1)$, all $\tau>0$ and all $n\in\N$:
\begin{align}
\label{eq:lypfastexp}
(1+2\lambda_\eps'\tau)\lyp(u_\tau^n,v_\tau^n)&\le \lyp(u_\tau^{n-1},v_\tau^{n-1})+\tau\eps M,
\end{align}
where $\lambda_\eps':=\lambda_0-L'\eps>0$.
\item Define, with the quantities from (a) and fixed, but arbitrary $\delta>0$:
\begin{align*}
&M':=\frac{M\eps_1}{2(\lambda_0-L'\eps_1)}>0 \qquad\text{and} \\ &T_0:=\max\left(0,\frac{1+2\delta}{2\lambda_\eps'}\log\frac{\ent(u^0,v^0)-\ent(u_\infty,v_\infty)+K\eps_1}{M'}\right)\ge 0,
\end{align*}
where $K>0$ is the constant from \eqref{eq:LE}. Then, there exists $\bar\tau>0$ such that for all $\eps\in(0,\eps_1)$, $\tau\in(0,\bar\tau]$ and $n\in\N$ with $n\tau\ge T_0$, one has
\begin{align}
\label{eq:lypbd}
\lyp(u_\tau^n,v_\tau^n)&\le 2M'.
\end{align}
\end{enumerate}
\end{proposition}

\begin{proof}
\begin{enumerate}[(a)]
\item We first estimate the last term appearing in $\dsp(u,v)$ from \eqref{eq:dsp}. Let $(u,v)\in\X\cap(W^{1,2}\times W^{1,2})$. By H\"older's inequality,
\begin{align}
\int_\Rdrei (u+v)|\dff\psi|^2\dd x&\le (\|u\|_{L^{3/2}}+\|v\|_{L^{3/2}})\|\dff\psi\|_{L^6}^2.\label{eq:hold}
\end{align}
The term involving the gradient of $\psi$ can be treated with the Hardy-Littlewood-Sobolev inequality (see for example \cite[Thm.~4.3]{lieb2001} or \cite[Lemma 3.1]{kmx2015}) which is applicable for Poisson's kernel $\krnl$:
\begin{align}
\|\dff\psi\|_{L^6}^2&\le C\|u-v\|_{L^2}^2\le 2C\|u\|_{L^2}^2+2C\|v\|_{L^2}^2,\label{eq:hard}
\end{align}
for some constant $C>0$. Combining \eqref{eq:hold}\&\eqref{eq:hard}, using $\|u\|_{L^1}=1=\|v\|_{L^1}$ again, the $L^p$-interpolation inequality yields for some $\beta,\beta'\in(0,1)$:
\begin{align*}
&\int_\Rdrei (u+v)|\dff(\psi-\psi_\infty)|^2\dd x\le 2\int_\Rdrei (u+v)|\dff\psi|^2\dd x+2\int_\Rdrei (u+v)|\dff\psi_\infty|^2\dd x\\
&\le 4C\big(\|u\|_{L^3}^\beta \|u\|_{L^3}^{2\beta'}+\|u\|_{L^3}^\beta \|v\|_{L^3}^{2\beta'}+\|v\|_{L^3}^\beta \|u\|_{L^3}^{2\beta'}+\|v\|_{L^3}^\beta \|v\|_{L^3}^{2\beta'}\big)\\&\quad+2\int_\Rdrei (u+v)|\dff\psi_\infty|^2\dd x\\
&\le C'(\|u\|_{L^3}^4+\|v\|_{L^3}^4+1),
\end{align*}
for some $C'>0$, by Young's inequality and thanks to finiteness of $\|\dff\psi_\infty\|_{L^\infty}$. Now, we apply \eqref{eq:L3D} and obtain 
\begin{align*}
\dsp(u,v)&\ge \left(1-\frac{\eps}{2}(1+C'')\right)\int_\Rdrei\big(u|\dff(2u+U+\eps\psi_\infty)|^2+v|\dff(2v+V-\eps\psi_\infty)|^2\big)\dd x\\&\quad-\eps M,
\end{align*}
for suitable $C''>0$ and $M>0$. For $\eps<\frac{2}{1+C''}$, we further conclude by \eqref{eq:lypest} that
\begin{align*}
\dsp(u,v)&\ge 2\lambda_\eps\left(1-\frac{\eps}{2}(1+C'')\right)\lyp(u,v)-\eps M.
\end{align*}
Insertion into \eqref{eq:fint} yields (a).
\item We first prove the following explicit estimate for all $\tau>0$ and $n\in\N\cup\{0\}$ by induction over $n$:
\begin{align}
\label{eq:lypind}
\begin{split}
\lyp(u_\tau^n,v_\tau^n)&\le (\ent(u^0,v^0)-\ent(u_\infty,v_\infty)+K\eps_1)(1+2\lambda_\eps'\tau)^{-n}\\&\quad+\frac{M\eps}{2\lambda_\eps'}(1-(1+2\lambda_\eps'\tau)^{-n}).
\end{split}
\end{align}
Indeed, the claim holds for $n=0$ thanks to \eqref{eq:LE}. If it holds for an arbitrary $n\in\N\cup\{0\}$, we obtain with \eqref{eq:lypfastexp}:
\begin{align*}
&\lyp(u_\tau^{n+1},v_\tau^{n+1})\le (1+2\lambda_\eps'\tau)^{-1}\lyp(u_\tau^n,v_\tau^n)+(1+2\lambda_\eps'\tau)^{-1}\tau\eps M\\
&\le (1+2\lambda_\eps'\tau)^{-(n+1)}(\ent(u^0,v^0)-\ent(u_\infty,v_\infty)+K\eps_1)\\&\quad+\frac{M\eps}{2\lambda_\eps'}(1+2\lambda_\eps'\tau)^{-1}(1-(1+2\lambda_\eps'\tau)^{-n})+(1+2\lambda_\eps'\tau)^{-1}\tau\eps M\\
&=(\ent(u^0,v^0)-\ent(u_\infty,v_\infty)+K\eps_1)(1+2\lambda_\eps'\tau)^{-(n+1)}+\frac{M\eps}{2\lambda_\eps'}(1-(1+2\lambda_\eps'\tau)^{-(n+1)}).
\end{align*}
Let now $\tau>0$ and $n\in\N$ with $n\tau\ge T_0$. Thanks to \eqref{eq:lypind}, for each $\delta>0$,
\begin{align*}
\lyp(u_\tau^n,v_\tau^n)&\le \frac{M\eps}{2\lambda_\eps'}+(\ent(u^0,v^0)-\ent(u_\infty,v_\infty)+K\eps_1)\exp\left(-\frac{n\tau}{\tau}\log(1+2\lambda_\eps'\tau)\right)
\\&\le (\ent(u^0,v^0)-\ent(u_\infty,v_\infty)+K\eps_1)\exp\left(-\frac{T_0}{\tau}\log(1+2\lambda_\eps'\tau)\right)+M'.
\end{align*}
Obviously, we obtain \eqref{eq:lypbd} in the case $\ent(u^0,v^0)-\ent(u_\infty,v_\infty)+K\eps_1\le M'$.
Consider the converse case. Since $\lim\limits_{s\to 0}\frac{\log(1+s)}{s}=1$, there exists $\bar s>0$ such that 
$\frac{\log(1+s)}{s}\ge \frac{1}{1+2\delta}$ for all $s\in(0,\bar s]$.
Henceforth, defining $\bar\tau:=\frac{\bar s}{2\lambda_0}$ yields
$\frac{\log(1+2\lambda_\eps'\tau)}{2\lambda_\eps'\tau}\ge \frac{1}{1+2\delta}$ for all $\tau\in(0,\bar\tau]$, and we arrive at the desired estimate by definition of $T_0$:
\begin{align*}
&\quad\lyp(u_\tau^n,v_\tau^n)\\&\le M'+(\ent(u^0,v^0)-\ent(u_\infty,v_\infty)+K\eps_1)\exp\left(-\log\frac{\ent(u^0,v^0)-\ent(u_\infty,v_\infty)+K\eps_1}{M'}\right)\\&= 2M'.\qedhere
\end{align*}
\end{enumerate}
\end{proof}

\subsection{Exponential convergence to equilibrium}
We are now able to prove -- for sufficiently large times -- a refined version of Proposition \ref{prop:lypbd}(a):
\begin{proposition}[Exponential estimate for $\lyp$]\label{prop:lypexp}
There exist constants $\bar\eps\in (0,\eps_1)$ and $\bar L>0$ such that for arbitrary $\delta>0$, there exists $\bar\tau>0$ such that for all $\eps\in (0,\bar\eps)$, $\tau\in(0,\bar\tau]$ and $n\in\N$ with $n\tau\ge T_0$, we have
\begin{align}\label{eq:lypexp}
(1+2\Lambda_\eps\tau)\lyp(u_\tau^n,v_\tau^n)&\le \lyp(u_\tau^{n-1},v_\tau^{n-1}),
\end{align}
with $\Lambda_\eps:=\lambda_0-\bar L\eps>0$ and $T_0$ as in Proposition \ref{prop:lypbd}(b).
\end{proposition}

\begin{proof}
We write $(u,v)$ instead of $(u_\tau^n,v_\tau^n)$ for the sake of clarity and consider the last term in $\dsp(u,v)$ once more. Using as in the proof of Proposition \ref{prop:lypbd}(a) the H\"older, Hardy-Littlewood-Sobolev and $L^p$-interpolation inequalities (cf. \eqref{eq:hold}\&\eqref{eq:hard}), we get for some $C,C'>0$ and $\beta\in(0,1)$:
\begin{align*}
&\quad\int_\Rdrei (u+v)|\dff(\psi-\psi_\infty)|^2\dd x\\
&=\int_\Rdrei ((u-u_\infty)+(v-v_\infty)+(u_\infty+v_\infty))|\dff(\psi-\psi_\infty)|^2\dd x\\
&\le C\|(u-u_\infty)-(v-v_\infty)\|_{L^2}^2\\&\quad\cdot\big(\|u-u_\infty\|_{L^2}^\beta \|u-u_\infty\|_{L^1}^{1-\beta}+\|v-v_\infty\|_{L^2}^\beta \|v-v_\infty\|_{L^1}^{1-\beta}+\|u_\infty+v_\infty\|_{L^{3/2}}\big)\\
&\le C\cdot 2\lyp(u,v)\cdot C'(1+\lyp(u,v))\le 2CC'(1+2M') \lyp(u,v),
\end{align*}
with Young's inequality, \eqref{eq:lypest} and \eqref{eq:lypbd}. Now, \eqref{eq:lypexp} follows thanks to \eqref{eq:fint}, for sufficiently small $\eps>0$.
\end{proof}

Finally, we prove Theorem \ref{thm:conv}.
\begin{proof}
Consider a vanishing sequence $(\tau_k)_{k\in\N}$ such that the corresponding sequence of discrete solutions $(u_{\tau_k},v_{\tau_k})_{k\in\N}$ converges to a weak solution to \eqref{eq:pnp_u}--\eqref{eq:poi}, in the sense of Theorem \ref{thm:exist}. Lower semicontinuity yields for all $t\ge 0$: \break $\lyp(u(t),v(t))\le \liminf\limits_{k\to\infty} \lyp(u_{\tau_k}(t),v_{\tau_k}(t))$. By \eqref{eq:LE} and the monotonicity of $\ent$ from Proposition \ref{prop:minmov}, one obtains after passage to $k\to\infty$ that
\begin{align}
\label{eq:lcbd}
\lyp(u(t),v(t))&\le \ent(u^0,v^0)-\ent(u_\infty,v_\infty)+K\eps_1\qquad\forall t\ge 0.
\end{align}
Iterating the estimate \eqref{eq:lypexp}, assuming without loss of generality that $k\in\N$ is sufficiently large, we get in the limit $k\to\infty$ that
\begin{align}
\label{eq:lcexp}
\lyp(u(t),v(t))&\le 2M'\exp(-2\Lambda_\eps (t-T_0))\qquad\forall t\ge T_0.
\end{align}
Actually, \eqref{eq:lcbd}\&\eqref{eq:lcexp} imply that $\lyp(u(t),v(t))\le A\exp(-2\Lambda_\eps t)$ for all $t\ge 0$, with some constant $A>0$, the particular structure of which remaining to be identified. Consider the case $\ent(u^0,v^0)-\ent(u_\infty,v_\infty)+K\eps_1\le M'$. Then $T_0=0$, so \eqref{eq:lcexp} holds for all $t\ge 0$. In the other case, combining \eqref{eq:lcbd}\&\eqref{eq:lcexp} yields for all $t\ge 0$:
\begin{align*}
\lyp(u(t),v(t))&\le \max(\ent(u^0,v^0)-\ent(u_\infty,v_\infty)+K\eps_1,2M')\exp(2\Lambda_\eps T_0)\exp(-2\Lambda_\eps t).
\end{align*}
We insert the definition of $T_0$ and use that $\Lambda_\eps\le\lambda_\eps'$ to find
\begin{align*}
\lyp(u(t),v(t))&\le \max\left((\ent(u_0,v_0)-\ent(u_\infty,v_\infty)+K\eps_1) ,2M'\right)\\&\quad\cdot\left(\frac{\ent(u_0,v_0)-\ent(u_\infty,v_\infty)+K\eps_1}{M'}\right)^{1+2\delta}\exp(-2\Lambda_\eps t).
\end{align*}
Combining both cases yields
\begin{align*}
\lyp(u(t),v(t))&\le \max(\ent(u_0,v_0)-\ent(u_\infty,v_\infty)+K\eps_1,2M')\\&\quad\cdot\max\left(1,\frac{\ent(u_0,v_0)-\ent(u_\infty,v_\infty)+K\eps_1}{M'}\right)^{1+2\delta}\exp(-2\Lambda_\eps t).
\end{align*}
Thus, we can find $\tilde C_\delta>0$ such that for all $t\ge 0$, the following holds:
\begin{align*}
\lyp(u(t),v(t))&\le \tilde C_\delta(\ent(u^0,v^0)-\ent(u_\infty,v_\infty)+1)^{2(1+\delta)}\exp(-2\Lambda_\eps t).
\end{align*}
From here, the desired exponential estimate \eqref{eq:conv} follows by means of \eqref{eq:LE} and \eqref{eq:subdiff}.
\end{proof}

\providecommand{\href}[2]{#2}
\providecommand{\arxiv}[1]{\href{http://arxiv.org/abs/#1}{arXiv:#1}}
\providecommand{\url}[1]{\texttt{#1}}
\providecommand{\urlprefix}{URL }


\begin{thebibliography}{10}

\bibitem{agueh2008}
\newblock M.~Agueh,
\newblock Rates of decay to equilibria for {$p$}-{L}aplacian type equations,
\newblock \emph{Nonlinear Anal.}, \textbf{68} (2008), 1909--1927,
\newblock \urlprefix\url{http://dx.doi.org/10.1016/j.na.2007.01.043}.

\bibitem{savare2008}
\newblock L.~Ambrosio, N.~Gigli and G.~Savar{\'e},
\newblock \emph{Gradient flows in metric spaces and in the space of probability
  measures},
\newblock 2nd edition,
\newblock Lectures in Mathematics ETH Z\"urich, Birkh\"auser Verlag, Basel,
  2008.

\bibitem{amt2000}
\newblock A.~Arnold, P.~Markowich and G.~Toscani,
\newblock On large time asymptotics for drift-diffusion-{P}oisson systems,
\newblock \emph{Transport Theory Statist. Phys.}, \textbf{29} (2000), 571--581,
\newblock \urlprefix\url{http://dx.doi.org/10.1080/00411450008205893}.

\bibitem{amtu2001}
\newblock A.~Arnold, P.~Markowich, G.~Toscani and A.~Unterreiter,
\newblock On convex {S}obolev inequalities and the rate of convergence to
  equilibrium for {F}okker-{P}lanck type equations,
\newblock \emph{Comm. Partial Differential Equations}, \textbf{26} (2001),
  43--100,
\newblock \urlprefix\url{http://dx.doi.org/10.1081/PDE-100002246}.

\bibitem{benabd2004}
\newblock N.~Ben~Abdallah, F.~M{\'e}hats and N.~Vauchelet,
\newblock A note on the long time behavior for the drift-diffusion-{P}oisson
  system,
\newblock \emph{C. R. Math. Acad. Sci. Paris}, \textbf{339} (2004), 683--688,
\newblock \urlprefix\url{http://dx.doi.org/10.1016/j.crma.2004.09.025}.

\bibitem{bilerdolbeault2000}
\newblock P.~Biler and J.~Dolbeault,
\newblock Long time behavior of solutions of {N}ernst-{P}lanck and
  {D}ebye-{H}\"uckel drift-diffusion systems,
\newblock \emph{Ann. Henri Poincar\'e}, \textbf{1} (2000), 461--472,
\newblock \urlprefix\url{http://dx.doi.org/10.1007/s000230050003}.

\bibitem{bilerdolbeault2001}
\newblock P.~Biler, J.~Dolbeault and P.~A. Markowich,
\newblock Large time asymptotics of nonlinear drift-diffusion systems with
  {P}oisson coupling,
\newblock \emph{Transport Theory Statist. Phys.}, \textbf{30} (2001), 521--536,
\newblock \urlprefix\url{http://dx.doi.org/10.1081/TT-100105936},
\newblock The Sixteenth International Conference on Transport Theory, Part II
  (Atlanta, GA, 1999).

\bibitem{carrillo2012}
\newblock A.~Blanchet, E.~A. Carlen and J.~A. Carrillo,
\newblock Functional inequalities, thick tails and asymptotics for the critical
  mass {P}atlak-{K}eller-{S}egel model,
\newblock \emph{J. Funct. Anal.}, \textbf{262} (2012), 2142--2230,
\newblock \urlprefix\url{http://dx.doi.org/10.1016/j.jfa.2011.12.012}.

\bibitem{blanchet2012}
\newblock A.~Blanchet and P.~Lauren{\c{c}}ot,
\newblock The parabolic-parabolic {K}eller-{S}egel system with critical
  diffusion as a gradient flow in {$\mathbb{R}^d,\ d\ge3$},
\newblock \emph{Comm. Partial Differential Equations}, \textbf{38} (2013),
  658--686,
\newblock \urlprefix\url{http://dx.doi.org/10.1080/03605302.2012.757705}.

\bibitem{blanchet2014}
\newblock A.~Blanchet, J.~A. Carrillo, D.~Kinderlehrer, M.~Kowalczyk,
  P.~Lauren{\c{c}}ot and S.~Lisini,
\newblock A hybrid variational principle for the {K}eller-{S}egel system in
  $\mathbb{R}^2$,
\newblock \emph{ESAIM: M2AN},
\newblock \urlprefix\url{http://dx.doi.org/10.1051/m2an/2015021},
\newblock Doi:10.1051/m2an/2015021.

\bibitem{carrillo2000}
\newblock J.~A. Carrillo and G.~Toscani,
\newblock Asymptotic {$L^1$}-decay of solutions of the porous medium equation
  to self-similarity,
\newblock \emph{Indiana Univ. Math. J.}, \textbf{49} (2000), 113--142,
\newblock \urlprefix\url{http://dx.doi.org/10.1512/iumj.2000.49.1756}.

\bibitem{carrillo2006}
\newblock J.~A. Carrillo, M.~Di~Francesco and G.~Toscani,
\newblock Intermediate asymptotics beyond homogeneity and self-similarity: long
  time behavior for {$u_t=\Delta\phi(u)$},
\newblock \emph{Arch. Ration. Mech. Anal.}, \textbf{180} (2006), 127--149,
\newblock \urlprefix\url{http://dx.doi.org/10.1007/s00205-005-0403-4}.

\bibitem{cmv2003}
\newblock J.~A. Carrillo, R.~J. McCann and C.~Villani,
\newblock Kinetic equilibration rates for granular media and related equations:
  entropy dissipation and mass transportation estimates,
\newblock \emph{Rev. Mat. Iberoamericana}, \textbf{19} (2003), 971--1018,
\newblock \urlprefix\url{http://dx.doi.org/10.4171/RMI/376}.

\bibitem{carrillo2006contractions}
\newblock J.~A. Carrillo, R.~J. McCann and C.~Villani,
\newblock Contractions in the 2-{W}asserstein length space and thermalization
  of granular media,
\newblock \emph{Arch. Ration. Mech. Anal.}, \textbf{179} (2006), 217--263,
\newblock \urlprefix\url{http://dx.doi.org/10.1007/s00205-005-0386-1}.

\bibitem{difwunsch2008}
\newblock M.~Di~Francesco and M.~Wunsch,
\newblock Large time behavior in {W}asserstein spaces and relative entropy for
  bipolar drift-diffusion-{P}oisson models,
\newblock \emph{Monatsh. Math.}, \textbf{154} (2008), 39--50,
\newblock \urlprefix\url{http://dx.doi.org/10.1007/s00605-008-0532-6}.

\bibitem{gianazza2009}
\newblock U.~Gianazza, G.~Savar{\'e} and G.~Toscani,
\newblock The {W}asserstein gradient flow of the {F}isher information and the
  quantum drift-diffusion equation,
\newblock \emph{Arch. Ration. Mech. Anal.}, \textbf{194} (2009), 133--220,
\newblock \urlprefix\url{http://dx.doi.org/10.1007/s00205-008-0186-5}.

\bibitem{jko1998}
\newblock R.~Jordan, D.~Kinderlehrer and F.~Otto,
\newblock The variational formulation of the {F}okker-{P}lanck equation,
\newblock \emph{SIAM J. Math. Anal.}, \textbf{29} (1998), 1--17,
\newblock \urlprefix\url{http://dx.doi.org/10.1137/S0036141096303359}.

\bibitem{juengelbook2001}
\newblock A.~J{\"u}ngel,
\newblock \emph{Quasi-hydrodynamic semiconductor equations},
\newblock Progress in Nonlinear Differential Equations and their Applications,
  41, Birkh\"auser Verlag, Basel, 2001,
\newblock \urlprefix\url{http://dx.doi.org/10.1007/978-3-0348-8334-4}.

\bibitem{kmx2015}
\newblock D.~{Kinderlehrer}, L.~{Monsaingeon} and X.~{Xu},
\newblock {A {W}asserstein gradient flow approach to
  {P}oisson-{N}ernst-{P}lanck equations}, 2015,
\newblock Preprint. arXiv:1501.04437.

\bibitem{kinderlehrer2009}
\newblock D.~Kinderlehrer and M.~Kowalczyk,
\newblock The {J}anossy effect and hybrid variational principles,
\newblock \emph{Discrete Contin. Dyn. Syst. Ser. B}, \textbf{11} (2009),
  153--176,
\newblock \urlprefix\url{http://dx.doi.org/10.3934/dcdsb.2009.11.153}.

\bibitem{laurencot2011}
\newblock P.~Lauren{\c{c}}ot and B.-V. Matioc,
\newblock A gradient flow approach to a thin film approximation of the {M}uskat
  problem,
\newblock \emph{Calc. Var. Partial Differential Equations}, \textbf{47} (2013),
  319--341,
\newblock \urlprefix\url{http://dx.doi.org/10.1007/s00526-012-0520-5}.

\bibitem{lieb2001}
\newblock E.~H. Lieb and M.~Loss,
\newblock \emph{Analysis}, vol.~14 of Graduate Studies in Mathematics,
\newblock 2nd edition,
\newblock American Mathematical Society, Providence, RI, 2001.

\bibitem{markowichbook1990}
\newblock P.~A. Markowich, C.~A. Ringhofer and C.~Schmeiser,
\newblock \emph{Semiconductor equations},
\newblock Springer-Verlag, Vienna, 1990,
\newblock \urlprefix\url{http://dx.doi.org/10.1007/978-3-7091-6961-2}.

\bibitem{matthes2009}
\newblock D.~Matthes, R.~J. McCann and G.~Savar{\'e},
\newblock A family of nonlinear fourth order equations of gradient flow type,
\newblock \emph{Comm. Partial Differential Equations}, \textbf{34} (2009),
  1352--1397,
\newblock \urlprefix\url{http://dx.doi.org/10.1080/03605300903296256}.

\bibitem{mccann1997}
\newblock R.~J. McCann,
\newblock A convexity principle for interacting gases,
\newblock \emph{Adv. Math.}, \textbf{128} (1997), 153--179,
\newblock \urlprefix\url{http://dx.doi.org/10.1006/aima.1997.1634}.

\bibitem{otto2001}
\newblock F.~Otto,
\newblock The geometry of dissipative evolution equations: the porous medium
  equation,
\newblock \emph{Comm. Partial Differential Equations}, \textbf{26} (2001),
  101--174,
\newblock \urlprefix\url{http://dx.doi.org/10.1081/PDE-100002243}.

\bibitem{villani2003}
\newblock C.~Villani,
\newblock \emph{Topics in optimal transportation}, vol.~58 of Graduate Studies
  in Mathematics,
\newblock American Mathematical Society, Providence, RI, 2003,
\newblock \urlprefix\url{http://dx.doi.org/10.1007/b12016}.

\bibitem{zinsl2012}
\newblock J.~Zinsl,
\newblock Existence of solutions for a nonlinear system of parabolic equations
  with gradient flow structure,
\newblock \emph{Monatsh. Math.}, \textbf{174} (2014), 653--679,
\newblock \urlprefix\url{http://dx.doi.org/10.1007/s00605-013-0573-3}.

\bibitem{zinsl2015}
\newblock J.~Zinsl,
\newblock A note on the variational analysis of the parabolic--parabolic
  {K}eller--{S}egel system in one spatial dimension,
\newblock \emph{C. R. Math. Acad. Sci. Paris}, \textbf{353} (2015), 849--854,
\newblock \urlprefix\url{http://dx.doi.org/10.1016/j.crma.2015.06.014}.

\bibitem{zinsl2014}
\newblock J.~Zinsl and D.~Matthes,
\newblock Exponential convergence to equilibrium in a coupled gradient flow
  system modeling chemotaxis,
\newblock \emph{Anal. PDE}, \textbf{8} (2015), 425--466,
\newblock \urlprefix\url{http://dx.doi.org/10.2140/apde.2015.8.425}.

\bibitem{zm2014}
\newblock J.~Zinsl and D.~Matthes,
\newblock Transport distances and geodesic convexity for systems of degenerate
  diffusion equations,
\newblock \emph{Calc. Var. Partial Differential Equations}, \textbf{advance
  online publication} (2015), 1--42,
\newblock \urlprefix\url{http://dx.doi.org/10.1007/s00526-015-0909-z}.

\end{thebibliography}
\end{document}